\documentclass[11pt]{amsart}

\usepackage[OT2,T1]{fontenc}
\usepackage{amsmath}
\usepackage{amssymb}
\usepackage{amscd}
\usepackage{euscript}

\usepackage{hyperref}
\usepackage[osf,sc]{mathpazo}

\newcommand{\isomto}{\overset{{}_\sim}{\rightarrow}}

 \def\Z{{\mathbf{Z}}} \def\1{{\mathbf{1}}}
\def\F{{\mathcal{F}}} \def\G{{\mathbf{G}}} 

\def\C{{\mathcal{C}}} \def\M{{\mathcal{M}}}

\def\I{{\mathcal{I}}}

\def\HH{{\mathbb{H}}}
\def\H{{\mathrm{H}}}

\def\PP{{\mathbf{P}}}
\def\AA{{\mathbf{A}}}

\def\O{{\mathcal{O}}}
\def\OXz{{\O_{X,z}^{{}^\wedge}}}

\def\Hom{{\rm Hom}}

\def\Ext{{\rm Ext}}

\def\id{{\rm id}}

\DeclareMathOperator{\coker}{coker}
\DeclareMathOperator{\Pic}{Pic}
\DeclareMathOperator{\End}{End}
\DeclareMathOperator{\Spec}{Spec}

\def\pr{{\rm pr}}

\def\phi{{\varphi}}

\theoremstyle{plain}
\newtheorem{theorem}{Theorem}
\newtheorem{corollary}{Corollary}
\newtheorem{lemma}{Lemma}
\newtheorem{proposition}{Proposition}

\theoremstyle{definition}
\newtheorem{definition}{Definition}

\newtheorem{remark}{Remark}

\begin{document}

\title[The Carlitz shtuka]
 {The Carlitz shtuka}

\author{Lenny Taelman}

\begin{abstract}
Recently we have used the Carlitz exponential map to define a finitely generated submodule of the Carlitz module having the right properties to be a function field analogue of the group of units in a number field. Similarly, we constructed a finite module analogous to the class group of a number field.

In this short note more algebraic constructions of these ``unit'' and ``class'' modules are given and they are related to  $\Ext$ modules in the category of shtukas.
\end{abstract}

\maketitle


\section{Introduction and statement of the main results}

\subsection{Notation}

Let $k$ be a finite field of $q$ elements. Without mention to the contrary schemes are
understood to be over $ \Spec k $ and tensor products over
$ k $.  

Let $ t $ be the standard coordinate on the projective line $\PP^1$ over $k$,
let $ F=k(t) $ the function field of $ \PP^1/k $ and let $ A =k[t] $ the ring of functions regular 
away from the ``point at infinity'' $\infty \in \PP^1$. 

Let $ X $ be a smooth projective geometrically connected curve over $ k $ and
$ X \to \PP^1 $ a surjective map. Denote the function field of $X$ by $K$. Let
$ Y \subset X $ be the inverse image of $ \Spec A = \PP^1 -\infty $.

\subsection{The Carlitz module}

\begin{definition}\label{defCarlitz}
The Carlitz module is the functor
\[
	C_0\colon \{ \Spec A\text{-schemes}\} \to \{A\text{-modules}\}
\] 
which associates to a scheme $S$ over $ \Spec A $ the $ A $-module $C(S)$
given by $C(S)= \Gamma(S, \O_S) $ as a $k$-vector space, with $A$-module structure 
\[
	\phi\colon A \to \End_k \Gamma(S, \O_S) \colon
	t \mapsto \left( c \mapsto tc + c^q \right).
\]
\end{definition}

The functor $ C_0 $ is in many ways an analogue of the multiplicative group
\[
	\G_m\colon \{\Spec \Z \text{-schemes}\} \to \{\Z\text{-modules}\}\colon 
		S \mapsto \Gamma(S,\O_S)^\times.
\]
Yet, in contrast with Dirichlet's unit theorem we have the following negative result:

\begin{proposition}[Poonen \cite{Poonen97}]
The $ A $-module $ C_0( Y ) $ is not finitely generated.\qed
\end{proposition}

\subsection{A construction using the Carlitz exponential}
In \cite{Taelman10} we have used the Carlitz exponential map to cut out a canonical finitely generated sub-$A$-module from $ C_0(Y) $. We recall and reformulate this construction.

A simple recursion shows that there is a unique power series $\exp x$ in $ F[[x]] $
which is of the form
\[
	\exp x = x + e_1x^q + e_2x^{q^2} + \cdots  
\]
and which satisfies
\begin{equation}\label{expfuneq}
	\exp (tx) = t\exp x + (\exp x)^q.
\end{equation}
This power series is called the \emph{Carlitz exponential}. It is entire and for every point $z$ of $X\backslash Y$ it defines an $A$-linear map $ \exp\colon K_z \to C_0( K_z ) $.

Note that $ U \mapsto C_0(U) $ defines a (Zariski) sheaf of $ A $-modules on $ Y $. We extend
it to a sheaf $C$ on $ X$ as follows:
\[
	C( U ) := 
	\left\{ ( c, ( \gamma_z )_z )
		\in C_0( U\cap Y ) \times \!\!\prod_{z\in U\backslash Y}\! K_z 
		\,\,\big|\,\,  \forall z\,  \exp \gamma_z = c \right\}.
\]
One verifies easily that this indeed defines a sheaf on $ X $. The main result of \cite{Taelman10} can be restated as follows:

\begin{proposition} \label{finiteness} $~$
\begin{enumerate}
\item $ \H^0( X, C ) $ is a finitely generated $ A $-module;
\item $ \H^1( X, C ) $ is a finite $ A $-module.
\end{enumerate}
\end{proposition}

In \S \ref{comparison} we will show how to deduce this result from \cite{Taelman10}.

In particular, the image of the restriction map $ C(X) \to C(Y) = C_0(Y) $ is a canonical
finitely generated submodule of $ C_0(Y) $, it is a Carlitz analogue of the group of units
in a number field. Similarly, $ \H^1( X, C ) $ is a Carlitz analogue of the class group of
a number field (this should of course be compared with the isomorphism
$\H^1( \O_L, \G_m ) = \Pic \O_L$).

We \emph{do} need to pass to the completed curve $X$ to get something interesting: By Poonen's theorem $\H^0( Y, C_0 ) $ is not finitely generated, and since $ C_0 \cong \O_Y $ as sheaves of abelian groups we have that $ \H^1( Y, C_0 ) = 0 $.

Unfortunately the above definition of the sheaf $C$ is analytic in nature, and it would be desirable to have a purely algebraic description of $C$. The aim of this paper is to provide such a description, as well as a more ``motivic'' interpretation of it. 

\subsection{An algebraic description of the sheaf $ C $}
For an integer $n$, denote by $ \O_{\PP^1}(n\infty) $ the sheaf of functions on $ \PP^1 $ that have a pole of order at most $n$ at $\infty$ and by $ \O_X(n\infty) $ its pullback over $X\to \PP^1$. 

\begin{theorem}\label{thm1}
There is a short exact sequence of sheaves of $A$-modules on $ X $
\begin{equation}\label{thm1seq}
	0 \longrightarrow \O_X \otimes A \overset{\partial}{\longrightarrow}
	\O_X(\infty) \otimes A \longrightarrow C \longrightarrow 0
\end{equation}
where
\[
	\partial\colon f\otimes a \mapsto  f\otimes ta - (tf+f^q) \otimes a.
\]
\end{theorem}

The proof of this theorem will be given in section \S \ref{sec1}.

\subsection{Interpretation in terms of shtukas}

The short exact sequence of Theorem \ref{thm1} can be reinterpreted in terms of shtukas. 

For any $k$-scheme $ S $ denote by $ S_A $ the base change of $ S $ to $ \Spec A $ and
by $ \tau_A\colon S_A \to S_A $ the base change of the $q$-th power Frobenius endomorphism 
$\tau\colon S \to S $.  

\begin{definition}
A (right) \emph{$A$-shtuka} on a $k$-scheme $ S $ is a diagram 
\[
	\M = \left[  
	\M \overset{\sigma}{\longrightarrow} \M' \overset{j}{\longleftarrow} \tau_A^\ast \M 
	\right]
\]
of quasi-coherent $ \O_{S_A} $-modules. 
\end{definition}

With the obvious notion of morphism, the shtukas on $ S $ form an $A$-linear
abelian category. In particular, if $ \M_1 $ and $\M_2$ are two shtukas on $ S $ then
the Yoneda extension groups $\Ext^i( \M_1, \M_2 ) $ are $ A $-modules.

We have a natural isomorphism of sheaves of $ \O_{X} $-modules
\[
	\tau^\ast \O_{X} \overset{{}_\sim}{\longrightarrow} \O_{X},
\]
and will identify source and target in what follows. 

If $\F$ is a coherent sheaf of $\O_X$-modules and $M$ an $A$-module we denote by
$\F \boxtimes M$ the coherent sheaf of $\O_{X\times \Spec(A)}$-modules
\[
	\F \boxtimes M = \pr_1^\ast \F \otimes_{\O_{X\times \Spec A }} \pr_2^\ast \tilde{M}
\]
where $\pr_1$ and $\pr_2$ denote the projections from $X\times \Spec{A}$ to $X$ and
$\Spec A$ respectively.

\begin{definition}
The \emph{unit shtuka} on $X$ is defined to be the shtuka
\[
	\1 = \left[
	\O_{X} \boxtimes A
	 \overset{1}{\longrightarrow}
	 \O_{X} \boxtimes A
	 \overset{1}{\longleftarrow}
	 \tau^\ast\O_{X} \boxtimes A
	\right].
\]
\end{definition}

\begin{definition}
We define the \emph{Carlitz shtuka} on $ X $ to be the sthuka
\[
	\C = \left[ 
		\O_{X} \boxtimes A
		\overset{ \sigma }{\longrightarrow}
		\O_{X}( \infty ) \boxtimes A
		\overset{ 1 }{ \longleftarrow }
		\tau^\ast\O_{X} \boxtimes A
	\right]
\]
with
\[
	\sigma = 1\otimes t - t\otimes 1.
\]
\end{definition}

The following is essentially a formal consequence of Theorem \ref{thm1}:

\begin{theorem} \label{thm2} There are natural isomorphisms
\[
	\Ext^i( \1, \C ) \overset{{}_\sim}{\longrightarrow} \H^{i-1}( X, C )
\]
for all $i$.
\end{theorem}

The proof will be given in \S \ref{sec2}.

\subsection{Remarks}

\begin{remark}
Our notion of shtuka is the same as the one in V.~Lafforgue \cite{Lafforgue09}. It is similar to the one used by Drinfeld \cite{Drinfeld77E} and L.~Lafforgue \cite{Lafforgue98}, but of a more arithmetic nature. Rather than compactifying the ``coefficients'' $\Spec A$ to a complete curve, we compactify the ``base'' $\AA^1_k$ to $\PP^1_k$.  
\end{remark}

\begin{remark}
Shtukas are function field toy models for (conjectural) mixed motives. The Carlitz shtuka $\C$ is an analogue of the Tate motive $\Z(1)$ and Theorem \ref{thm2} should be compared with the isomorphisms
\[
	\Ext^1_X(\1,\Z(1))=\Gamma(X,\O_X^\times) \,\,\text{ and }\,\,
	\Ext^2_X(\1,\Z(1))=\Pic X
\]
from motivic cohomology, see for example \cite[p. 25]{Mazza06}.
\end{remark}

\begin{remark}
In the ($\infty$-adic) ``class number formula'' proven in \cite{Taelman10b}, the $A$-modules $\H^0(X,C)$ and $\H^1(X,C)$ play a role analogous to the groups of units and the class group in the classical class number formula. In the guise of $\Ext^1(\1,\C)$ and
$\Ext^2(\1,\C)$ they play a similar role in V.~Lafforgue's result 
\cite{Lafforgue09} on  ($v$-adic, $v\neq \infty$) special values.
\end{remark}

\begin{remark}
For any $m$ there is a natural isomorphism $\tau^\ast\O_X(m\infty)\isomto\O_X(qm\infty)$. 
So in the definition of the Carlitz shtuka one could twist both line bundles with $\O_X(-n\infty)$ for some $n\geq 0$ to obtain
\[
		\O_{X}(-n\infty) \boxtimes A
		\overset{ \sigma }{\longrightarrow}
		\O_{X}( (1-n)\infty ) \boxtimes A
		\overset{ 1 }{ \longleftarrow }
		\tau^\ast\O_{X}(-n\infty) \boxtimes A.
\]
The same results with the same proofs hold for this shtuka. We have chosen $n=0$ in our definition somewhat arbitrarily, distinguishing it from the other choices only by its
minimality.
\end{remark}

\begin{remark}
We have treated in this note only a very special case. One should try to obtain similar results for higher rank Drinfeld modules over general Drinfeld rings $A$, and even for the
abelian $t$-modules of Anderson \cite{Anderson86}. Unfortunately it seems that these generalizations are not without difficulty, and even for higher rank Drinfeld modules it is not clear to me what the precise statement should be.
\end{remark}

\subsection{Acknowledgements}

This work has been inspired by
work of Anderson and Thakur \cite{Anderson90}, Woo \cite{Woo95}, Papanikolas and Ramachandran
\cite{Papanikolas03}, and V.~Lafforgue \cite{Lafforgue09}. The author is grateful to David Goss for his feedback and constant encouragement, and to the referee for several useful suggestions. 

The author is supported by a VENI Grant from the Netherlands Organization for Scientific Research (NWO). Part of the research leading to this paper was carried out at the Ecole Polytechnique F\'ed\'erale de Lausanne. 

\section{The cohomology of the sheaf $C$}\label{comparison}

In this section we show how the modules $\H^0(X,C)$ and $\H^1(X,C)$ compare with the modules studied in \cite{Taelman10}. We recall the main result of \emph{loc. cit.} Consider the map
\[
	\delta\colon
		C_0(Y) \times \!\!\prod_{z\in X\backslash Y}\! K_z
		\to 
		\prod_{z\in X\backslash Y}\! C_0(K_z)
	\colon
		(c, (\gamma_z)_z ) \mapsto ( c - \exp \gamma_z )_z.
\]
\begin{theorem}[\cite{Taelman10}] \label{oldthm}
$\ker \delta$ is a finitely generated $A$-module and
$\coker \delta$ is a finite $A$-module.\qed
\end{theorem}

We now show that $\H^0(X,C)$ and $\H^1(X,C)$ coincide with the modules
$\ker \delta$ (``$\exp^{-1}C(R)$'' in the notation of \emph{loc. cit.}) and 
$\coker \delta$ (``$H_R$'') above, and hence that Proposition \ref{finiteness} follows
from Theorem \ref{oldthm}.

\begin{lemma}\label{lemmaC}
There is an exact sequence of $ A $-modules
\[
	0 \longrightarrow
		C( X )
	\longrightarrow
		C_0(Y) \times \!\!\prod_{z\in X\backslash Y}\! K_z 
	\overset{\delta}{\longrightarrow}
		\prod_{z\in X\backslash Y}\! C_0(K_z)
	\longrightarrow
		\H^1( X, C )
	\longrightarrow 0.
\]
\end{lemma}

\begin{proof}
Denote by $i\colon Y\to X$ and by
$ i_z\colon \{z\} \to X$ the inclusions of $Y$ and the points $z$ in $ X $. Then
the following sequence of sheafs on $ X $ is exact:
\begin{equation}\label{sheafshort}
	0 \longrightarrow
		C
	\longrightarrow
		i_\ast C_0 \times \!\!\prod_{z\in X\backslash Y}\! i_{z,\ast} K_z
	\longrightarrow
		\prod_{z\in X\backslash Y}\! i_{z,\ast} C_0(K_z)
	\longrightarrow 0.
\end{equation}
(Here the middle map is the difference of the natural map and the map induced by $\exp$.) Left exactness follows from the definition of $ C $. For right exactness, one uses the fact
that for all $z\in X\backslash Y$ we have $C_0(K_z) = C_0(K) + \exp K_z  $ (which follows, for example, from Corollary \ref{corexpinv} below).

Note that $\H^1(X,i_\ast C_0)=\H^1(Y,C_0)=0$ so that the desired exact sequence is precisely
the long exact sequence of cohomology obtained from taking global sections in
(\ref{sheafshort}). 
\end{proof}

\section{Proof of Theorem \ref{thm1}}\label{sec1}

\subsection{Away from $\infty$}

Let $ R $ be an $ A $-algebra. Denote by $\alpha$ the $A$-linear map
\[
	\alpha\colon R\otimes A \to C_0(R)\colon r \otimes a \mapsto \phi(a)(r).
\]

\begin{proposition}\label{propaway}
The sequence of $ A $-modules
\begin{equation}\label{away}
	0 \longrightarrow R\otimes A \overset{\partial}{\longrightarrow} R\otimes A 
	\overset{\alpha}{\longrightarrow} C_0(R) \longrightarrow 0
\end{equation}
is exact.
\end{proposition}

\begin{proof} Straightforward. \end{proof}

In particular this provides the desired short exact sequence of sheaves (\ref{thm1seq}) on the
affine curve $Y\subset X$. In the following paragraphs we will extend it to the whole of $ X $.

\subsection{Inversion of the exponential map} Let $z\in X\backslash Y$ and let
$|\cdot|$ be an absolute value on $K_z$ inducing the $z$-adic topology, so in particular $|t|>1$. 

\begin{lemma} For all $ x \in K_z $ with $ |x| < |t|^{q/(q-1)} $ we have $ |\exp x-x| < |x| $.
\end{lemma}

\begin{proof}
Write $ \exp x = \sum_{i=0}^\infty e_ix^{q^i} $. It follows from (\ref{expfuneq}) and
from $e_0=1$ that for all $ i $ we have $|e_i| = |t|^{-iq^i}$. From this one deduces that for
all $i>0$ and all $x$ with $|x|<|t|^{q/(q-1)}$ the inequality $|e_ix^{q^i}| < |x|$
holds. Hence $|\exp x -x|=|\sum_{i>0} e_ix^{q^i} | < |x| $, as claimed.
\end{proof}

\begin{corollary}\label{corexpinv}
For all $m\leq 1$ the exponential map restricts to a $k$-linear isomorphism
$t^m\OXz \to t^m \OXz$.\qed
\end{corollary}

We denote its inverse by $\log$.

\subsection{Near $\infty$}

Let $z\in X\backslash Y$. Consider the $A$-linear map
\[
	\lambda\colon t\OXz\otimes A \to K_z\colon f\otimes a \mapsto a\log f.
\]

\begin{proposition}\label{propLie}
The sequence of $A$-modules
\begin{equation}\label{extLie}
	0 \longrightarrow
	\OXz\otimes A 
		\overset{\partial}{\longrightarrow}
	t\OXz\otimes A 
		\overset{\lambda}{\longrightarrow}
	K_z \longrightarrow
	0
\end{equation}
is exact.
\end{proposition}

\begin{proof}
Denote by $\mu$ the multiplication  map
\[
	\mu\colon t\OXz\otimes A \to K_z\colon f\otimes a \mapsto af.
\]
Using the identity (\ref{expfuneq})
one verifies that the diagram
\[
\begin{CD}
	0 @>>> \OXz\otimes A 
	 	@>{1\otimes t-t\otimes 1}>>
		t\OXz\otimes A 
		@>{\mu}>> K_z @>>> 0 \\
	@. @VV{\exp\otimes\id}V  @VV{\exp\otimes\id}V @VV{\id}V @. \\
	0 @>>> \OXz\otimes A 
	 	@>{\partial}>>
		t\OXz\otimes A
		@>{\lambda}>> K_z @>>> 0
\end{CD}
\]
commutes. The vertical arrows are isomorphisms by Corollary \ref{corexpinv} and since the top sequence is exact the same holds for the bottom sequence.
\end{proof}

\subsection{Conclusion}

It is now a purely formal matter to deduce Theorem \ref{thm1} from Propositions \ref{propaway}
and \ref{propLie}:

\begin{proof}[Proof of Theorem \ref{thm1}]
Clearly the map
\[
	\O_X\otimes A \overset{\partial}{\longrightarrow}
		\O_X(\infty) \otimes A
\]
is injective. We need to construct an isomorphism $\coker \partial  \isomto C$.

For every open $U\subset X$ and every integer $m$ we have an exact sequence
\[
	0 \longrightarrow
	\O_X(m\infty)(U) \longrightarrow
	\O_X(U\cap Y) \times t^m \prod_{z} \OXz  \overset{\delta}{\longrightarrow}
	\prod_z K_z 
\]
where the products range over $z\in U\backslash Y$ and where $\delta(f,g):=f-g$. If moreover
$U$ is affine then $\delta$ is surjective and we obtain a short exact sequence which
we denote by $E(m)$. 

Now, for an affine $U$, consider the map of exact sequences
\[
	\partial\colon E(0)\otimes A \to E(1) \otimes A.
\]
It is injective in all three positions. Using 
(\ref{away}) and (\ref{extLie}) one sees that the quotient is isomorphic with a short exact
sequence
\[
	0 \longrightarrow
	(\coker \partial)(U) \longrightarrow
	C_0(U\cap Y) \times \prod_z K_z \longrightarrow 
	\prod_z C_0(K_z) \longrightarrow
	0,
\]
the last map being $(c,f)\mapsto c-\exp f $. This provides an isomorphism
$(\coker \partial)(U) \isomto C(U)$ for every affine open $ U\subset X$, and clearly these
glue to an isomorphism of sheaves. This proves Theorem \ref{thm1}. 
\end{proof}

\section{Proof of Theorem \ref{thm2}} \label{sec2}

Let $ S $ be a $ k $-scheme. For any $ \O_{S_A} $-module $ \F $ we 
denote by $\tau$ the canonical isomorphism of $ S_A $-sheaves
\[
	\tau\colon \F \longrightarrow \tau_A^\ast \F
\]
which is $ A $-linear but generally \emph{not} $ \O_{S_A} $-linear. If
\[
	\M = \left[  
	\M \overset{\sigma}{\longrightarrow} \M' \overset{j}{\longleftarrow} \tau^\ast \M 
	\right]
\]
is a shtuka on $ S $ then we denote by $ \M^\bullet $ the complex of  $ S_A $-sheaves
\[
	 \M \overset{\partial}{\longrightarrow} \M'
\]
in degrees $0$ and $1$, with $\partial=\sigma-j\circ\tau$.

The following Proposition can be found implicitly in \cite{Lafforgue09}.  

\begin{proposition}\label{proplaff}
For all $ i $ and all $\M$ there are natural isomorphisms 
\[
	\Ext^i( \1, \M ) = \HH^i( S_A, \M^\bullet ),
\]
functorial in $ \M $ and in $ S $.
\end{proposition}

Before giving a proof, we first deduce Theorem \ref{thm2} from this proposition.

\begin{proof}[Proof of Theorem \ref{thm2}]
Applying the Proposition to $S=X$ and $\M=\C$ we find
\[
	\Ext^i( \1, \C ) = \HH^i( X_A,
		\O_X\boxtimes A \overset{\partial}{\longrightarrow} \O_X(\infty)\boxtimes A ).
\]
The latter is isomorphic with
\[
	\HH^i( X, \O_X\otimes A \overset{\partial}{\longrightarrow} \O_X(\infty)\otimes A )
\]
which by Theorem \ref{thm1} is isomorphic with $\H^{i-1}( X, C ) $.
\end{proof}

\begin{proof}[Proof of Proposition \ref{proplaff}]
We will first establish a canonical isomorphism for $ i=0 $, and then conclude the
general case by a purely formal argument.

A homomorphism $\1 \to \M $ of shtukas on $ S $ is a commutative diagram
\[
\begin{CD}
	\O_{S_A} @>1>> \O_{S_A} @<1<< \tau^\ast \O_{S_A} \\
	@VVfV	@VV{f'}V @VV{\tau^\ast f}V \\
	\M @>\sigma>> \M' @<j<< \tau^\ast \M
\end{CD}
\]
Clearly the homomorphism is uniquely determined by $f \in \Gamma( S_A, \M ) $, and
an $ f \in \Gamma( S_A, \M ) $ extends to a homomorphism of shtukas if and only if
$ \partial f = 0 $. So we obtain an exact sequence
\[
	0 \longrightarrow \Hom( \1, \M )
	  \longrightarrow \Gamma( S_A, \M )
	  \overset{\partial}{\longrightarrow} \Gamma( S_A, \M' )
\]
and hence an isomorphism 
\[
	\Hom( \1, \M ) = \HH^0( S_A, \M^\bullet ).
\]
	  
Now any shtuka
\[
	\I = \left[  
	\I \overset{\sigma}{\longrightarrow} \I' \overset{j}{\longleftarrow} \tau^\ast \I 
	\right]
\]
with $ \I $ and $ \I' $ injective $\O_{S_A}$-modules is an injective object in the category of
shtukas on $ S $. So we can find an injective resolution $ \I^\bullet $ of
the shtuka $ \M $ such that the resulting double complex $ \I^{\bullet\bullet} $ is
an injective resolution of the complex $ \M^\bullet $. We obtain a canonical isomorphism
\[
	\Ext^i( \1, \M ) = \HH^i( S_A, \M^\bullet )
\]
for all $ i $.
\end{proof}

\bibliographystyle{plain}
\bibliography{../master}

\end{document}